\documentclass[11pt]{amsart}
\usepackage{amsmath, amssymb, mathtools, graphicx, subfigure, paralist, enumitem}
\usepackage[paper=a4paper,left=3cm,right=3cm,top=3cm,bottom=3cm]{geometry}

\newtheorem{theorem}{Theorem}[section]

\newtheorem{corollary}[theorem]{Corollary}

  \theoremstyle{definition}

  \theoremstyle{remark}
\newtheorem{remark}[theorem]{Remark}

\numberwithin{equation}{section}

\newcommand{ \dl }{ \delta_\textsc{l} }
\newcommand{ \tl }{ \vartheta_\textsc{l} }
\newcommand{ \vol }{ \operatorname{vol} }
\newcommand{ \defeq }{ \coloneqq }

\begin{document}

\title[Improved Bounds on Sidon Sets via Lattice Packings of Simplices]
      {Improved Bounds on Sidon Sets via\\Lattice Packings of Simplices}

\author{Mladen~Kova\v{c}evi\'{c} and Vincent~Y.~F.~Tan}

\address{Department of Electrical \& Computer Engineering,
         National University of Singapore.}
\email{mladen.kovacevic@nus.edu.sg}
\address{Department of Electrical \& Computer Engineering and Department of Mathematics,\linebreak
         National University of Singapore.}
\email{vtan@nus.edu.sg}
\thanks{The work is supported by the Singapore Ministry of Education (MoE) Tier 2 grant R-263-000-B61-112.}

\subjclass[2010]{Primary: 05B10, 05B40, 11B13, 11B75, 11B83, 11H31, 52C17;
                 Secondary: 94B65, 94B75, 20K99.}

\date{September 21, 2017.}

\keywords{$ B_h $ set, Sidon set, additive basis, $ h $-basis,
lattice packing, packing density, lattice covering, simplex.}

\begin{abstract}
  A $ B_h $ set (or Sidon set of order $ h $) in an Abelian group $ G $ is any
subset $ \{b_0, b_1, \ldots,b_{n}\} $ of $ G $ with the property that all the
sums $ b_{i_1} + \cdots + b_{i_h} $ are different up to the order of the summands.
Let $ \phi(h,n) $ denote the order of the smallest Abelian group containing a
$ B_h $ set of cardinality $ n + 1 $.
It is shown that
\[
  \lim_{h \to \infty} \frac{ \phi(h,n) }{ h^n } = \frac{1}{n! \ \dl(\triangle^n)} ,
\]
where $ \dl(\triangle^n) $ is the lattice packing density of an $ n $-simplex
in Euclidean space.
This determines the asymptotics exactly in cases where this density is known
($ n \leq 3 $) and gives improved bounds on $ \phi(h,n) $ in the remaining
cases.
The corresponding geometric characterization of bases of order $ h $ in finite
Abelian groups in terms of lattice coverings by simplices is also given.
\end{abstract}

\maketitle

\section{Introduction}

\subsection{$ \boldsymbol{B_h} $ sets}
\label{sec:Bh}

Let $ G $ be a finite additive Abelian group, and $ h, n $ positive integers.
A set $ B = \{ b_0, b_1, \ldots,b_{n} \} \subseteq G $ is said to be a $ B_h $ set
(or Sidon%
\footnote{Sidon's consideration of $ B_2 $ sets \cite{sidon}, motivated by applications
in the theory of Fourier series, is one of the first occurrences of this notion in the
literature.}
set of order $ h $) if all the sums $ b_{i_1} + \cdots + b_{i_h} $ with
$ 0 \leq i_1 \leq \cdots \leq i_h \leq n $ are different.
If $ B $ is a $ B_h $ set, then so is $ B - b_0 = \{ 0, b_1-b_0, \ldots,b_{n}-b_0 \} $,
and vice versa;
therefore, we shall assume in what follows that $ b_0 = 0 $.
With this convention $ B $ is a $ B_h $ set if and only if the sums
$ \alpha_1 b_{1} + \cdots + \alpha_n b_{n} $ are different for every choice of the
coefficients $ \alpha_i \in \mathbb{Z} $, $ \alpha_i \geq 0 $, $ \sum_{i=1}^n \alpha_i \leq h $
(here $ \alpha_i b_i $ denotes the sum of $ \alpha_i $ copies of the element $ b_i \in G $).

We should note that most of the works on $ B_h $ sets are focused either on the
group of integers $ G = \mathbb{Z} $ or the group of residues modulo $ m $, i.e., the
cyclic group $ G = \mathbb{Z}_m \defeq \mathbb{Z} / m \mathbb{Z} $ \cite{bose+chowla, chen, jia, obryant}.
For what we intend to discuss here it is convenient to consider general---not necessarily
cyclic---finite Abelian groups.
Such a generalization is also of interest in applications to coding theory \cite{derksen}.

Let $ \phi(h,n) $ denote%
\footnote{In \cite{jia, kovacevic} the symbol $ \phi(h,n+1) $ was used
to denote this quantity, $ n + 1 $ being the cardinality of the $ B_h $ set in question.
We shall, however, use $ n $ as an argument of this function, both to simplify the notation
and to indicate the fact that one of the elements can always be fixed to zero.}
the order (i.e., the cardinality) of the smallest Abelian group containing a $ B_h $
set of cardinality $ n + 1 $.
Understanding the interplay between $ h $, $ n $, and $ \phi(h, n) $, particularly in
various asymptotic regimes, is one of the basic problems in the study of $  B_h $ sets
\cite{chen, jia, obryant}.
We are interested here in the problem of understanding
\begin{equation}
\label{eq:mainproblem}
  \phi(h,n)  \quad \text{when} \quad h \to \infty .
\end{equation}

The following bounds are known:
\begin{equation}
\label{eq:boundsh}
  \frac{ (2n)! }{ 2^n (n!)^3 } (h-2n+2)^n  <  \phi(h,n)  \leq  (h+1)^n ,
\end{equation}
where the left-hand inequality holds for $ 0 \leq 2n - 2 \leq h $ \cite{kovacevic},
and the right-hand inequality holds for all positive $ h, n $ \cite{jia}.
In particular,
\begin{equation}
\label{eq:limBhkfixed}
  \frac{ (2n)! }{ 2^n (n!)^3 }  \leq  \lim_{h\to\infty} \frac{ \phi(h,n) }{ h^n }  \leq  1 .
\end{equation}
The main result of the present paper is a geometric characterization of the problem
\eqref{eq:mainproblem} in terms of lattice packings of simplices in $ \mathbb{R}^n $.
As a consequence of this observation, the upper bound in \eqref{eq:limBhkfixed} will
be improved from $ 1 $ to a rapidly decreasing function of $ n $,
and the exact value of the limit in \eqref{eq:limBhkfixed} will be determined
for $ n \leq 3 $.
(The fact that the limit in \eqref{eq:limBhkfixed} exists will also be proven in
Section \ref{sec:main}.)

Several results which are known to hold in a different asymptotic regime---$ h $ fixed
and $ n \to \infty $---should also be mentioned.
The following bound, valid for $ 1 \leq h/2 \leq n+1 $, is to the best of our knowledge
the best known%
\footnote{It was first derived for cyclic groups and in a slightly different form in \cite{jia}
(for even $ h $) and in \cite{chen} (for odd $ h $).
The bound in \eqref{eq:Bhhfixed} and its validity in all finite Abelian groups were
proven in \cite{kovacevic}.}
\cite{chen, jia, kovacevic}:
\begin{equation}
\label{eq:Bhhfixed}
  \phi(h,n)  >  \frac{(n + 1 - \lceil h/2 \rceil)^h}{\lceil h/2 \rceil ! \lfloor h/2 \rfloor !} .
\end{equation}
The construction of Bose and Chowla \cite{bose+chowla}, which generalizes Singer's
construction \cite{singer} of $ B_2 $ sets to arbitrary $ h $, asserts that
$ \phi(h,n) \leq n^h + n^{h-1} + \cdots + 1 $ for $ n $ a prime power, and so
\begin{equation}
\label{eq:limBhhfixed}
  \frac{1}{\lceil h/2 \rceil ! \lfloor h/2 \rfloor !}  \leq  \liminf_{n\to\infty} \frac{ \phi(h,n) }{n^h}  \leq  1 .
\end{equation}
Determining the exact value of the $ \liminf $ in \eqref{eq:limBhhfixed} for $ h \geq 3 $
remains an outstanding open problem in the field.

For an overview of known results on $ B_h $ sets and related objects, and an extensive
list of references, see \cite{obryant}.

\subsection{Lattice packings}
\label{sec:packings}

Let $ \mathcal L $ be a full-rank lattice in $ \mathbb{R}^n $, i.e.,
$ {\mathcal L} = \big\{ \alpha_1 {\bf v}_1 + \cdots + \alpha_n {\bf v}_n : \alpha_i \in \mathbb{Z}, i = 1, \ldots, n \big\} $
for some set of linearly independent vectors
$ \{ {\bf v}_1, \ldots, {\bf v}_n \} \subset \mathbb{R}^n $
(we say that $ {\mathcal L} $ is generated by the vectors $ {\bf v}_1, \ldots, {\bf v}_n $).
The determinant of $ {\mathcal L} $ is defined as
$ \det({\mathcal L}) \defeq |\det({\bf v}_1, \ldots, {\bf v}_n)| $ and
represents the volume of the fundamental parallelotope
$ P({\bf v}_1, \ldots, {\bf v}_n) \defeq \big\{ t_1 {\bf v}_1 + \cdots + t_n {\bf v}_n : t_i \in [0,1], i = 1, \ldots, n \big\} $.
Let $ K \subset \mathbb{R}^n $ be a compact convex set with non-empty interior.
$ (K, {\mathcal L}) $ is said to be a lattice packing in $ \mathbb{R}^n $ if, for
every $ {\bf x}, {\bf y} \in {\mathcal L} $, $ {\bf x} \neq {\bf y} $, the translates
$ K + {\bf x} $ and $ K + {\bf y} $ have no interior points in common.
The density of such a packing is defined as
$ \delta(K, {\mathcal L}) \defeq \vol(K) / \det({\mathcal L}) $, where $ \vol(K) $
denotes the volume of $ K $, and the lattice packing density of the body $ K $ is then
$ \dl(K) \defeq \sup_{\mathcal L} \delta(K, {\mathcal L}) $.
The supremum here is taken over all lattices in $ \mathbb{R}^n $ and is always
attained for some lattice $ {\mathcal L}^* $, i.e., $ \dl(K) = \delta(K, {\mathcal L}^*) $
\cite[Thm 30.1, p.\ 445]{gruber}.
For a very nice and extensive account of the theory of lattices and packing problems,
see \cite{gruber+lekk}.

We are interested here in lattice packings of the body
$ \triangle^n \defeq \big\{ {\bf x} \in \mathbb{R}^n : x_i \geq 0,\linebreak \sum_{i=1}^n x_i \leq 1 \big\} $
(here and hereafter $ {\bf x} $ stands for $ (x_1, \ldots, x_n) $).
$ \triangle^n $ is a simplex of volume $ \vol(\triangle^n) = \frac{1}{n!} $.
For $ n \leq 3 $, exact values of the lattice packing densities of $ \triangle^n $ are
known \cite[p.\ 249]{gruber+lekk}, \cite{hoylman}:
\begin{equation}
\label{eq:dlnsmall}
  \dl(\triangle^1) = 1,  \qquad  \dl(\triangle^2) = \frac{2}{3},  \qquad  \dl(\triangle^3) = \frac{18}{49} ,
\end{equation}
while in higher dimensions one can only give bounds on this quantity at this point \cite[p.\ 72]{rogers}:
\begin{equation}
\label{eq:dlsimplex}
  \frac{ 2 (n!)^2 }{ (2n)! }  \leq  \dl(\triangle^n)  <  \frac{ 2^n (n!)^2 }{ (2n)! } .
\end{equation}
A better lower bound%
\footnote{Lower bounds better than the one in \eqref{eq:dlsimplex} can also be obtained
by concrete constructions in some cases.
For example, the trivial packing $ (\triangle^n, \mathbb{Z}^n) $ has density $  1 / n! $,
which is larger than $ 2 (n!)^2 / (2n)! $ for $ n = 4, 5 $.
We shall omit from the discussion such special cases, however, except those for which the
exact value of $ \dl(\triangle^n) $ is known; see \eqref{eq:dlnsmall}.}
is known for $ n \to \infty $ \cite{fejes-toth}: 
\begin{equation}
\label{eq:dlupper}
  \frac{ (\log 2 + o(1))\ n\ (n!)^2 }{ (2n)! }  \leq  \dl(\triangle^n) .
\end{equation}

Packings in discrete spaces are defined in a similar way.
For example, if $ {\mathcal L} \subseteq \mathbb{Z}^n $ is a sublattice of $ \mathbb{Z}^n $
and $ S \subset \mathbb{Z}^n $ a finite set, we say that $ (S, {\mathcal L}) $ is a
lattice packing in $ \mathbb{Z}^n $ if the sets $ S + {\bf x} $ and $ S + {\bf y} $ are
\emph{disjoint} for every $ {\bf x}, {\bf y} \in {\mathcal L} $, $ {\bf x} \neq {\bf y} $.
Abusing notation slightly we denote the density of such a packing also by
$ \delta(S, {\mathcal L}) \defeq |S| / \det({\mathcal L}) $, and we let
$ \dl(S) \defeq \sup_{\mathcal L} \delta(S, {\mathcal L}) $, the supremum being
taken over all sublattices of $ \mathbb{Z}^n $.

We shall also need a discrete analogue of $ \triangle^n $, namely
$ \triangle^n_h \defeq \big\{ {\bf x} \in \mathbb{Z}^n : x_i \geq 0, \sum_{i=1}^n x_i \leq h \big\} $.
The set $ \triangle^n_h $ is a discrete simplex of ``sidelength'' $ h $, dimension $ n $,
and cardinality $ \binom{h + n}{n} $.

\section{Main results}
\label{sec:main}

The problem of constructing $ B_h $ sets of given cardinality in ``small'' Abelian
groups is closely related to that of constructing dense lattice packings of simplices
in Euclidean spaces.
This observation is the main result that we wish to report here and is stated
precisely in the following theorem.

A notational convention: For $ a \in \mathbb{R} $, $ S \subset \mathbb{R}^n $,
we let $ a S = \{ a {\bf x} : {\bf x} \in S \} $.

\begin{theorem}
\label{thm:main}
  For every fixed $ n \geq 1 $ and $ \epsilon > 0 $,
\begin{equation}
\label{eq:mainl}
  \frac{ 1 }{ n! \ \dl(\triangle^n) } h^n  \leq  \phi(h,n)  <  \frac{ (1 + \epsilon) }{ n!\ \dl(\triangle^n) } h^n ,
\end{equation}
the lower bound being valid for every $ h \geq 1 $, and the upper bound for
$ h \geq h_0(n,\epsilon) $.\linebreak
Consequently,
\begin{equation}
\label{eq:main}
  \lim_{h\to\infty} \frac{ \phi(h,n) }{ h^n } = \frac{ 1 }{ n!\ \dl(\triangle^n) } .
\end{equation}
\end{theorem}
\begin{proof}
  The proof builds on the geometric interpretation of $ B_h $ sets given in
\cite[Thm 2.1]{kovacevic}, stating that
\begin{inparaenum}
\item[(a)]
if $ \{0, b_1, \ldots, b_n\} $ is a $ B_h $ set in an Abelian group $ G $, then
there is a lattice packing $ \left(\triangle^n_h, {\mathcal L}\right) $ in $ \mathbb{Z}^n $
with $ G \cong \mathbb{Z}^n/{\mathcal L} $, $ |G| = \det({\mathcal L}) $, and
\item[(b)]
if $ (\triangle^n_h, {\mathcal L}) $ is a lattice packing in $ \mathbb{Z}^n $, then
the group $ \mathbb{Z}^n / {\mathcal L} $ contains a $ B_h $ set of cardinality $ n + 1 $.
\end{inparaenum}
In particular, the lattice packing density of the discrete simplex $ \triangle^n_h $ is
\begin{equation}
\label{eq:dlnh}
  \dl(\triangle^n_h) = \frac{\binom{h + n}{n}}{ \phi(h,n) } .
\end{equation}

To prove the left-hand inequality in \eqref{eq:mainl}, observe that any packing
$ \left(\triangle^n_h, {\mathcal L}\right) $ in $ \mathbb{Z}^n $ induces a packing
$ \left( h \triangle^n, {\mathcal L} \right) $ in $ \mathbb{R}^n $, and that the
density of the latter cannot exceed $ \dl(\triangle^n) $ by definition.
If $ \mathcal L $ is optimal, in the sense that $ \det({\mathcal L}) = \phi(h,n) $,
we get
\begin{equation}
\label{eq:philower}
  \phi(h,n)  =     \det({\mathcal L})
             \geq  \frac{ \vol(h\triangle^n) }{ \dl(\triangle^n) }
             =     \frac{ h^n/n! }{ \dl(\triangle^n) } .
\end{equation}

We now prove the right-hand inequality in \eqref{eq:mainl}.
This statement is, by \eqref{eq:dlnh} and the fact that
$ \lim_{h \to \infty} \frac{1}{h^n} \binom{h + n}{n} = \frac{1}{n!} $, equivalent to
the following: for every $ \varepsilon' \in (0,1) $ and $ h $ large enough,
$ \dl(\triangle^n_h) > (1-\varepsilon') \dl(\triangle^n) $.
So let $ \varepsilon' \in (0,1) $ and choose $ \varepsilon \in (0,1) $ so that
$ 1 - \varepsilon' < (1 - \varepsilon)^n $.
Let $ (\triangle^n, {\mathcal L}^*) $ be an optimal lattice packing in $ \mathbb{R}^n $.
In the discretization argument that follows, we shall need the simplices in the packing
to be bounded away from each other, so consider instead the packing $ ((1-\varepsilon) \triangle^n, {\mathcal L}^*) $
of density $ (1-\varepsilon)^n \dl(\triangle^n) $ (see Figure \ref{fig:packing}).
Suppose that $ {\mathcal L}^* $ is generated by the vectors $ {\bf v}_1, \ldots, {\bf v}_n $.
Let $ {\bf v}_{i;h} $ be the vector in $ \frac{1}{h} \mathbb{Z}^n $ obtained by
rounding each coordinate of $ {\bf v}_i $ to the nearest point in $ \frac{1}{h} \mathbb{Z} $
(breaking ties arbitrarily),
and let $ {\mathcal L}_h $ be the lattice generated by $ {\bf v}_{1;h}, \ldots, {\bf v}_{n;h} $.
Then clearly $ \lim_{h \to \infty} {\bf v}_{i;h} = {\bf v}_i $ and so
$ \lim_{h \to \infty} {\mathcal L}_h = {\mathcal L}^* $ in the natural topology on the space
of lattices \cite[Ch.\ 17.1]{gruber+lekk}.
Continuity properties of fundamental domains of lattices \cite{groemer} then imply that
$ ((1-\varepsilon) \triangle^n, {\mathcal L}_h) $ is a packing in $ \mathbb{R}^n $ for
all $ h \geq h_0(n, \varepsilon) $.
This further implies that
$ \left( (1-\varepsilon) \triangle^n \cap \frac{1}{h} \mathbb{Z}^n, {\mathcal L}_h \right) $
is a packing in $ \frac{1}{h} \mathbb{Z}^n $, which is equivalent to saying that
$ \left( (1-\varepsilon) h \triangle^n \cap \mathbb{Z}^n, h {\mathcal L}_h \right) $ is a
packing in $ \mathbb{Z}^n $.
Since $ \lim_{h \to \infty} \det({\mathcal L}_h) = \det({\mathcal L}^*) $, the density
of the latter satisfies
\begin{equation}
  \lim_{h \to \infty} \frac{ \left| (1-\varepsilon) h \triangle^n \cap \mathbb{Z}^n \right| }{ \det(h {\mathcal L}_h) }
  = \frac{ \vol((1-\varepsilon) \triangle^n) }{ \det({\mathcal L}^*) }
  = (1-\varepsilon)^n \dl(\triangle^n) .
\end{equation}
We have thus established the existence of a family of packings of $ \triangle^n_h $ in
$ \mathbb{Z}^n $ with asymptotic density
$ (1-\varepsilon)^n \dl(\triangle^n) > (1-\varepsilon') \dl(\triangle^n) $.
Therefore, for $ h $ large enough, $ \dl(\triangle^n_h) >  (1-\varepsilon') \dl(\triangle^n) $,
as claimed.

From this and \eqref{eq:philower} we conclude that
\begin{equation}
\label{eq:dlcont}
  \lim_{h\to\infty} \dl(\triangle^n_h) = \dl(\triangle^n) ,
\end{equation}
which is equivalent to \eqref{eq:main}.
\end{proof}

\begin{figure}
\centering
  \includegraphics[width=0.75\columnwidth]{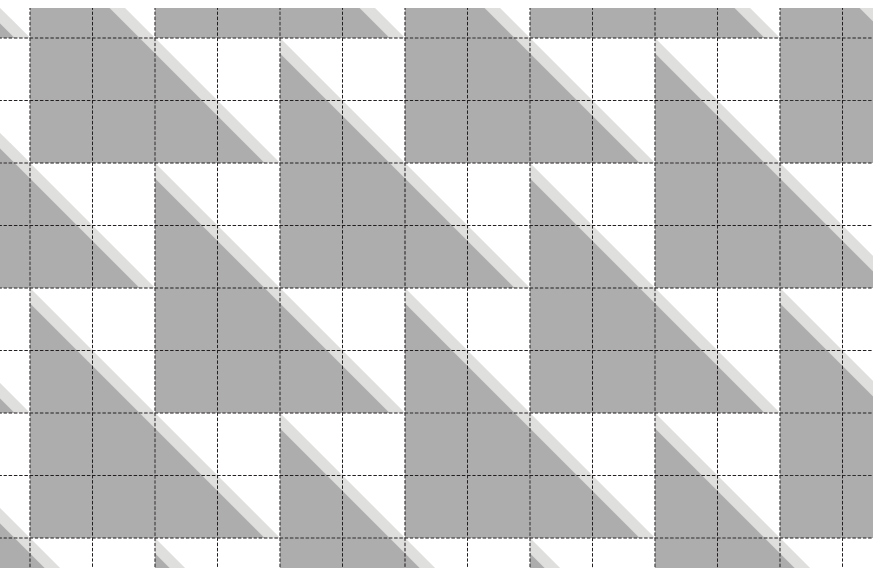}
  \caption{The optimal lattice packing of simplices $ (4 \triangle^2, {\mathcal L}^*) $
	         in $ \mathbb{R}^2 $ (light gray), the packing $ ((1-\varepsilon) 4 \triangle^2, {\mathcal L}^*) $ (dark gray),
           and the $ \mathbb{Z}^2 $ grid.}
\label{fig:packing}
\end{figure}%

\begin{remark}
  Analogous statements can be made in many other scenarios.
For example, lattice packings of discrete cross-polytopes
$ \Diamond^n_r \defeq \big\{ {\bf x} \in \mathbb{Z}^n : \sum_{i=1}^n |x_i| \leq r \big\} $
in $ \mathbb{Z}^n $ are of interest in coding theory as they represent linear
codes of radius $ r $ under the $ \ell_1 $ metric and are closely related to codes
in the so-called Lee metric ($ \ell_1 $ metric on the torus $ \mathbb{Z}_m^n $) 
\cite[Ch.\ 10]{roth}.
Their asymptotic behavior when $ r \to \infty $ can, similarly to \eqref{eq:dlcont},
be expressed as
\begin{equation}
  \lim_{r\to\infty} \dl(\Diamond^n_r) = \dl(\Diamond^n) ,
\end{equation}
where $ \Diamond^n \defeq \big\{ {\bf x} \in \mathbb{R}^n : \sum_{i=1}^n |x_i| \leq 1 \big\} $
is the $ n $-dimensional cross-polytope in $ \mathbb{R}^n $.
\end{remark}

Using the known facts about lattice packings of simplices stated in
\eqref{eq:dlnsmall}--\eqref{eq:dlupper}, we obtain the following corollary.

\begin{corollary}
\label{thm:maincor}
  For every fixed $ n \geq 1 $ and $ \epsilon > 0 $,
\begin{equation}
\label{eq:newlowerh}
  \frac{ (2n)! }{ 2^n (n!)^3 } h^n  <  \phi(h,n)  <  (1 + \epsilon) \frac{ (2n)! }{ 2 (n!)^3 } h^n ,
\end{equation}
the lower bound being valid for every $ h \geq 1 $, and the upper bound for
$ h \geq h_0(n,\epsilon) $.\linebreak
Furthermore, for $ n \leq 3 $,
\begin{equation}
\label{eq:nsmall}
  \lim_{h\to\infty} \frac{ \phi(h,1) }{ h }   = 1 ,            \qquad
  \lim_{h\to\infty} \frac{ \phi(h,2) }{ h^2 } = \frac{3}{4} ,  \qquad
  \lim_{h\to\infty} \frac{ \phi(h,3) }{ h^3 } = \frac{49}{108} ,
\end{equation}
and, for $ n \geq 4 $,
\begin{equation}
\label{eq:boundsnew}
  \frac{ (2n)! }{ 2^n (n!)^3 }  <  \lim_{h\to\infty} \frac{ \phi(h,n) }{h^n}  \leq  \frac{ (2n)! }{ 2 (n!)^3 } .
\end{equation}
For $ n \to \infty $ we have
\begin{equation}
\label{eq:boundsnew2}
  \lim_{h\to\infty} \frac{ \phi(h,n) }{h^n}  \leq  \frac{ (2n)! }{ (\log 2 + o(1))\ n\ (n!)^3 } = {\mathcal O}((4 e)^n n^{-n-2}) .
\end{equation}
\end{corollary}

Comparing \eqref{eq:boundsh} and \eqref{eq:newlowerh} we see that the lower bound
on $ \phi(h,n) $ is slightly improved, while the improvement of the upper bound is
significant.
In particular, the upper bound on $ \lim_{h\to\infty} \phi(h,n) / h^n $ is improved
from $ 1 $ to a rapidly decreasing function of $ n $ (see \eqref{eq:limBhkfixed} and
\eqref{eq:boundsnew2}).

For $ n \in \{1,2\} $ one can in fact give exact expressions for $ \phi(h,n) $:
$ \phi(h,1) = h + 1 $ and $ \phi(2r,2) = 3 r^2 + 3 r + 1 $, $ \phi(2r-1,2) = 3 r^2 $
(the case $ n = 1 $ is trivial, and the case $ n = 2 $ follows from the existence
of certain tilings of $ \mathbb{Z}^2 $; see Section \ref{sec:tilings}).

\section{Additional remarks and open problems}

\subsection{Equivalent packings and perfect $ \boldsymbol{B_h} $ sets}
\label{sec:tilings}

It is easy to show that if $ (\triangle^n_h, {\mathcal L}) $ is a packing in
$ \mathbb{Z}^n $, then so is $ (\triangle^n_{r} - \triangle^n_{t}, {\mathcal L}) $
for any $ r, t \geq 0 $ with $ r + t = h $, and vice versa.
Here $ \triangle^n_{r} - \triangle^n_{t} \defeq
\big\{ {\bf x} - {\bf y} : {\bf x} \in \triangle^n_{r}, {\bf y} \in \triangle^n_{t} \big\} $.

In this sense, packings of discrete simplices $ \triangle^n_{2r} $ are equivalent%
\footnote{This is a discrete version of the central symmetrization argument in geometry
\cite[p.\ 443]{gruber}.}
to packings of the sets $ \triangle^n_{r} - \triangle^n_{r} $.
It turns out \cite{kovacevic} that there exist perfect lattice packings, i.e.,
lattice tilings of $ \mathbb{Z}^n $ by $ \triangle^n_{r} - \triangle^n_{r} $, when
\begin{enumerate}
\item[1.)]  $ n \in \{1, 2\} $, $ r \geq 1 $, and
\item[2.)]  $ n $ a prime power, $ r = 1 $.
\end{enumerate}
Consequently, in these cases we have an exact expression%
\footnote{The expression for $ |\triangle^n_{r} - \triangle^n_{t}| $ is given in
\cite[Lem.\ 1.2]{kovacevic} ($ \triangle^n_{r} - \triangle^n_{t} $ is denoted by
$ S_n(r,t) $ there).}
for $ \phi(2r, n) $: $ \phi(2r, n) = \linebreak|\triangle^n_{r} - \triangle^n_{r}| $.
The corresponding $ B_{2r} $ sets might therefore be called \emph{perfect} $ B_{2r} $ sets.
It is an open problem to (dis)prove that these are the only cases when such sets exist
\cite{kovacevic}.
The corresponding question for tilings by discrete cross-polytopes $ \Diamond^n_r $
(perfect codes in $ \mathbb{Z}^n $ under $ \ell_1 $ metric) is known as the Golomb--Welch
conjecture and has inspired a significant amount of research since it was originally
published in \cite{golomb+welch}.

Similarly, one could define perfect $ B_{2r-1} $ sets as those that correspond to
lattice tilings by $ \triangle^n_{r} - \triangle^n_{r-1} $.
It can be verified directly that such sets exist for
\begin{enumerate}
\item[1.)] $ n \in \{1, 2\} $, $ r \geq 1 $ (see Figure \ref{fig:tiling}), and
\item[2.)] $ n \geq 1 $, $ r = 1 $.
\end{enumerate}
Again, whether these are the only cases is not known.

It is known that there are no perfect $ B_h $ sets of cardinality $ n + 1 \geq 4 $
for $ h $ large enough \cite[Thm 3.5]{kovacevic}.

\begin{figure}
\centering
  \includegraphics[width=0.75\columnwidth]{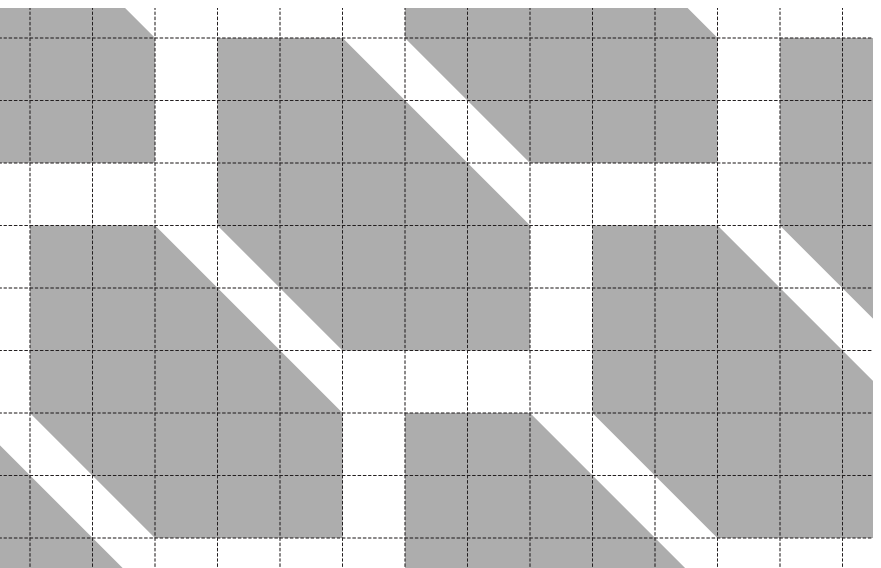}
  \caption{Lattice tiling of $ \mathbb{Z}^2 $ by the set $ \triangle^2_{3} - \triangle^2_{2} $.}
\label{fig:tiling}
\end{figure}%

\subsection{Asymptotics}

It follows from the discussion in the previous subsection that Theorem \ref{thm:main}
and Corollary \ref{thm:maincor} can be rephrased in terms of the lattice packing densities
of sets $ \triangle^n_{r} - \triangle^n_{t} $.
In particular, for all $ n \geq 3 $,
\begin{equation}
  \lim_{r \to \infty}  \dl(\triangle^n_{r} - \triangle^n_{r}) 
   = \dl(\triangle^n - \triangle^n)
   < 1 ,
\end{equation}
where $ \triangle^n - \triangle^n $ is the difference body of the simplex $ \triangle^n $
(hexagon for $ n = 2 $, cuboctahedron for $ n = 3 $ \cite{hoylman}).

Is it true that $ \limsup_{n \to \infty} \dl(\triangle^n - \triangle^n) < 1 $?
(This is equivalent to an asymptotic improvement of the upper bound in \eqref{eq:dlsimplex}.)
If so, can the limits be interchanged to conclude
$ \lim_{r \to \infty} \limsup_{n \to \infty} \dl(\triangle^n_{r} - \triangle^n_{r}) < 1 $?
A positive answer to the latter would give an improved lower bound in \eqref{eq:limBhhfixed}
and would imply non-existence results for perfect $ B_h $ sets.

\subsection{Cyclic groups}

As we mentioned in Section \ref{sec:Bh}, the literature on $ B_h $ sets in finite
groups is mostly focused on the cyclic case, $ G = \mathbb{Z}_m $.
It is worthwhile investigating whether Theorem \ref{thm:main} remains valid if we
restrict ourselves to cyclic groups only.
Namely, let $ \phi_\text{c}(h,n) $ denote the order of the smallest \emph{cyclic}
group containing a $ B_h $ set of cardinality $ n + 1 $.
Since $ \phi_\text{c}(h,n) \geq \phi(h,n) $, we have by Theorem \ref{thm:main}
\begin{equation}
\label{eq:cycliclim}
  \lim_{h\to\infty} \frac{ \phi_\text{c}(h,n) }{ h^n } \geq \frac{ 1 }{ n!\ \dl(\triangle^n) } .
\end{equation}
Does equality hold in \eqref{eq:cycliclim} for every $ n $?

\subsection{Bases of order $ \boldsymbol h $ and lattice coverings by simplices}

Let $ G $ be a finite Abelian group, as before.
A subset $ C = \{c_0, c_1, \ldots, c_n\} \subseteq G $ is said to be a basis of order
$ h $ (or $ h $-basis) \cite[Sec.\ I.1]{sequences} of $ G $ if every element of the
group can be expressed as $ c_{i_1} + \cdots + c_{i_h} $ for some
$ 0 \leq i_1 \leq \cdots \leq i_h \leq n $.
These objects are natural covering analogues of $ B_h $ sets.
In the following we state some results pertaining to their geometric interpretation,
in analogy to those obtained for $ B_h $ sets.
Note that, if $ C $ is an $ h $-basis, then so is $ C - c_0 = \{ 0, c_1 - c_0, \ldots, c_n - c_0 \} $,
and vice versa; we can therefore assume that $ c_0 = 0 $.
Note also that $ C = \{ 0, c_1, \ldots, c_n \} $ is an $ h $-basis for $ G $ if and
only if
\begin{equation}
\label{eq:basis}
  G = \Big\{ \alpha_1 c_1 + \cdots + \alpha_n c_n : \boldsymbol{\alpha} \in \triangle^n_h \Big\} .
\end{equation}

For a lattice $ \mathcal L $ and a convex body $ K $ in $ \mathbb{R}^n $, $ (K, {\mathcal L}) $
is said to be a lattice covering of $ \mathbb{R}^n $ if
$ \bigcup_{{\bf x} \in {\mathcal L}} (K + {\bf x}) = \mathbb{R}^n $.
The density of such a covering is defined as
$ \vartheta(K, {\mathcal L}) \defeq \vol(K) / \det({\mathcal L}) $,
and the lattice covering density of the body $ K $ is then
$ \tl(K) \defeq \inf_{\mathcal L} \vartheta(K, {\mathcal L}) $.
The infimum here is taken over all lattices in $ \mathbb{R}^n $ and is attained
for some $ {\mathcal L}^* $, i.e., $ \tl(K) = \vartheta(K, {\mathcal L}^*) $
\cite[Thm 31.1, p.\ 456]{gruber}.
In the case of the simplex $ \triangle^n $, the exact value of the lattice covering
density is known only for $ n = 1, 2 $ \cite[p.\ 249]{gruber+lekk}:
\begin{equation}
\label{eq:coveringexact}
  \tl(\triangle^1) = 1 ,  \qquad  \tl(\triangle^2) = \frac{3}{2} ,
\end{equation}
while for $ n \geq 3 $ we have the following bounds:
\begin{equation}
\label{eq:coveringbounds}
  1 + 2^{-(3n + 7)}  \leq  \tl(\triangle^n)  \leq  n^{\log_2 \log_2 n + c} ,
\end{equation}
where $ c $ is some absolute constant (for the upper bound see \cite[p.\ 19]{rogers};
the lower bound was recently derived in \cite{xue+zong}).
The definitions for coverings of $ \mathbb{Z}^n $ with a finite set $ K \subset \mathbb{Z}^n $
are similar, with $ \vol(K) $ replaced by $ |K| $.

The following claim is an instance of the familiar group--theoretic formulations of
lattice packing/tiling/covering problems \cite{stein+szabo} (see also, e.g., \cite{stein84}
and the references therein).
It states that $ h $-bases in Abelian groups and lattice coverings of $ \mathbb{Z}^n $ by
discrete simplices are essentially equivalent notions.

\begin{theorem}
\label{thm:Ch}
  If $ C = \{ 0, c_1, \ldots, c_n \} $ is an $ h $-basis for an Abelian group $ G $,
then $ (\triangle^n_h, {\mathcal L}) $ is a covering of $ \mathbb{Z}^n $, where
$ {\mathcal L} = \big\{ {\bf x} \in \mathbb{Z}^n : \sum_{i=1}^{n} x_i c_i = 0 \big\} $,
and $ G $ is isomorphic to $ \mathbb{Z}^n / {\mathcal L} $.
Conversely, if $ (\triangle^n_h, {\mathcal L}') $ is a lattice covering of $ \mathbb{Z}^n $,
then the group $ \mathbb{Z}^n / \mathcal{L}' $ contains an $ h $-basis of cardinality
at most $ n + 1 $.
\end{theorem}
\begin{proof}
  Let $ C $ be an $ h $-basis of $ G $, and
$ {\mathcal L} = \big\{ {\bf x} \in \mathbb{Z}^n : \sum_{i=1}^{n} x_i c_i = 0 \big\} $.
We need to show that $ \bigcup_{{\bf x} \in {\mathcal L}} ( \triangle^n_h + {\bf x} ) = \mathbb{Z}^n $,
i.e., that every vector $ {\bf y} \in \mathbb{Z}^n $ is contained in some translate
of the simplex $ \triangle^n_h $ by a vector from the lattice $ {\mathcal L} $.
Take an arbitrary $ {\bf y} \in \mathbb{Z}^n $, and suppose that $ \sum_{i=1}^{n} y_i c_i = a \in G $.
Since $ C $ is an $ h $-basis, we know that there exists a vector $ \boldsymbol{\alpha} \in \triangle^n_h $
such that $ \sum_{i=1}^{n} \alpha_i c_i = a $ (see \eqref{eq:basis}).
Then consider the vector $ {\bf x} = {\bf y} - \boldsymbol{\alpha} $.
Clearly, $ \sum_{i=1}^{n} x_i c_i = 0 $, i.e., $ {\bf x} \in {\mathcal L} $, and
$ {\bf y} = {\bf x} + \boldsymbol{\alpha} \in {\bf x} + \triangle^n_h $, proving
that $ (\triangle^n_h, {\mathcal L}) $ is a covering of $ \mathbb{Z}^n $.
Furthermore, it is easily checked that the mapping $ [{\bf y}] \mapsto \sum_{i = 1}^n y_i c_i $
is an isomorphism between the groups $ \mathbb{Z}^n / {\mathcal L} $ and $ G $,
where $ [{\bf y}] \defeq {\bf y} + {\mathcal L} $ are the cosets of the lattice
$ {\mathcal L} $ (elements of $ \mathbb{Z}^n / {\mathcal L} $).

To prove the converse statement, suppose that we are given a lattice covering
$ (\triangle^n_h, {\mathcal L}') $ of $ \mathbb{Z}^n $, meaning that an arbitrary
point $ {\bf y} \in \mathbb{Z}^n $ can be written as $ {\bf y} = {\bf x} + \boldsymbol{\alpha} $,
$ {\bf x} \in {\mathcal L}' $, $ \boldsymbol{\alpha} \in \triangle^n_h $.
This implies that $ [{\bf y}] = [\boldsymbol{\alpha}] = \sum_{i=1}^n \alpha_i [{\bf e}^{(i)}] $,
where $ {\bf e}^{(i)} $ is the unit vector having a $ 1 $ at the $ i^\text{th} $ coordinate
and zeros elsewhere.
In other words, every element $ [{\bf y}] \in \mathbb{Z}^n / {\mathcal L}' $ can be
written in the form $ [{\bf y}] = \sum_{i=1}^n \alpha_i [{\bf e}^{(i)}] $ for some
$ \boldsymbol{\alpha} \in \triangle^n_h $.
By \eqref{eq:basis}, this  means that $ C' \defeq \big\{ [{\bf 0}], [{\bf e}^{(1)}], \ldots, [{\bf e}^{(n)}] \big\} $
is an $ h $-basis in the group $ \mathbb{Z}^n / {\mathcal L}' $.
(Note that the cardinality of $ C' $ may be smaller than $ n + 1 $ as some of the
cosets in it may coincide.)
\end{proof}

We can now state the covering analogue of Theorem \ref{thm:main}.
Let $ \psi(h,n) $ denote the size of the \emph{largest} Abelian group containing an
$ h $-basis of size $ n + 1 $.

\begin{theorem}
  For every fixed $ n \geq 1 $,
\begin{equation}
\label{eq:limpsi}
  \lim_{h\to\infty} \frac{ \psi(h,n) }{ h^n } = \frac{ 1 }{ n!\ \tl(\triangle^n) } .
\end{equation}
\end{theorem}
\begin{proof}
  Theorem \ref{thm:Ch} implies that the lattice covering density of the discrete
simplex $ \triangle^n_h $ is
\begin{equation}
\label{eq:discretecoverdens}
  \tl(\triangle^n_h) = \frac{ \binom{h + n}{n} }{ \psi(h,n) } .
\end{equation}
Hence, to prove the theorem it is enough to show that
$ \lim_{h\to\infty} \tl(\triangle^n_h) = \tl(\triangle^n) $, or equivalently that, for
every $ \varepsilon' \in (0,1) $ and $ h $ large enough,
\begin{equation}
\label{eq:contcoverdens}
  (1 - \varepsilon') \tl(\triangle^n)  <  \tl(\triangle^n_h)  <  (1 + \varepsilon') \tl(\triangle^n) .	
\end{equation}

We first prove the left-hand side of \eqref{eq:contcoverdens}.
Observe the body $ C^n_h \defeq \bigcup_{{\bf x} \in \triangle^n_h} \big([0, 1]^n + {\bf x}\big) \subset \mathbb{R}^n $.
In words, $ C^n_h $ is the union of unit cubes $ [0, 1]^n $ translated to the
points of the discrete simplex $ \triangle^n_h $.
It is not hard to see that if $ \left(\triangle^n_h, {\mathcal L}\right) $ is a covering
of $ \mathbb{Z}^n $, then $ \left( C^n_h, {\mathcal L} \right) $ is a covering of $ \mathbb{R}^n $.
This implies that $ \left( (h + n) \triangle^n, {\mathcal L} \right) $ is also a covering of
$ \mathbb{R}^n $ because $ C^n_h \subset (h + n) \triangle^n $.
Now suppose that $ \left(\triangle^n_h, {\mathcal L}\right) $ is optimal, meaning that
$ \det({\mathcal L}) = \psi(h,n) $.
Then the density of the induced covering $ \left( (h + n) \triangle^n, {\mathcal L} \right) $
of $ \mathbb{R}^n $ is
\begin{equation}
\label{eq:tlinduced}
  \vartheta((h + n) \triangle^n, {\mathcal L})
    =     \frac{(h + n)^n / n!}{\det({\mathcal L})}
    =     \frac{(h + n)^n / n!}{\psi(h,n)}
    \geq  \tl(\triangle^n) ,
\end{equation}
where the last inequality is simply a consequence of the definition of $ \tl $.
The left-hand inequality in \eqref{eq:contcoverdens} now follows from \eqref{eq:discretecoverdens},
\eqref{eq:tlinduced}, and the fact that $ \binom{h + n}{n} > (1 - \varepsilon') (h + n)^n / n! $
for every $ \varepsilon' \in (0,1) $ and large enough $ h $ (recall that $ n $ is fixed).

The proof of the right-hand inequality in \eqref{eq:contcoverdens} is analogous to the
proof of the corresponding statement for packings (see the proof of Theorem \ref{thm:main}).
Take an optimal covering $ \left(\triangle^n, {\mathcal L}^*\right) $ of $ \mathbb{R}^n $,
and consider the covering $ \left((1+\varepsilon)\triangle^n, {\mathcal L}^*\right) $ of
density $ (1+\varepsilon)^n \tl(\triangle^n) $, where $ \varepsilon $ is chosen so that
$ (1+\varepsilon)^n < 1+\varepsilon' $.
Define a sequence of lattices $ {\mathcal L}_h \subseteq \frac{1}{h} \mathbb{Z}^n $
such that $ {\mathcal L}_h \to {\mathcal L}^* $ as $ h \to \infty $.
Then, for $ h $ large enough, $ ((1+\varepsilon) \triangle^n, {\mathcal L}_h) $ is a
covering of $ \mathbb{R}^n $ and
$ \left( (1+\varepsilon) h \triangle^n \cap \mathbb{Z}^n, h {\mathcal L}_h \right) $
is a covering of $ \mathbb{Z}^n $.
The density of the latter satisfies
\begin{equation}
  \lim_{h \to \infty} \frac{ \left| (1+\varepsilon) h \triangle^n \cap \mathbb{Z}^n \right| }{ \det(h {\mathcal L}_h) }
  = (1+\varepsilon)^n \tl(\triangle^n) 
	< (1 + \varepsilon') \tl(\triangle^n) ,
\end{equation}
which proves the right-hand inequality in \eqref{eq:contcoverdens}.
\end{proof}

In particular, from \eqref{eq:coveringexact} and \eqref{eq:limpsi} we get
$ \lim_{h\to\infty} \frac{ \psi(h,2) }{ h^2 } = \frac{1}{3} $.
Hence, the largest group containing an $ h $-basis of cardinality $ 3 $ is about
$ 9/4 $ times smaller than the smallest group containing a $ B_h $ set of the same
size, for large $ h $.

Finally, we note that group-theoretic analogues of lattice coverings by
$ \triangle^n_{r} - \triangle^n_{t} $ can be defined in a similar way, by the condition
\begin{equation}
  G = \Big\{ \alpha_1 c_1 + \cdots + \alpha_n c_n : \boldsymbol{\alpha} \in \triangle^n_{r} - \triangle^n_{t} \Big\}
\end{equation}
(The familiar notion of $ (v, k, \lambda) $-difference sets \cite{designs} is
a particular instance of objects of this sort, see \cite[Thm 4.1]{kovacevic}.)
Explicit constructions of such coverings are left as another problem for future
investigation.

\section*{Acknowledgment}

The authors would like to thank the anonymous referees for a thorough review and
many helpful comments on the original version of the manuscript, as well as for
pointing out the work \cite{xue+zong}.

\bibliographystyle{amsplain}

\begin{thebibliography}{99}

\bibitem{designs}
   T. Beth, D. Jungnickel, and H. Lenz,
   \emph{Design Theory}, 2nd ed.,
   Cambridge University Press, 1999.
\bibitem{bose+chowla}
   R. C. Bose and S. Chowla,
   ``Theorems in the Additive Theory of Numbers,''
   \emph{Comment. Math. Helv.}, vol. 37, no. 1, pp. 141--147, Dec. 1962.
\bibitem{chen}
   S. Chen,
   ``On the Size of Finite Sidon Sequences,''
   \emph{Proc. Amer. Math. Soc.}, vol. 121, no. 2, pp. 353--356, Jun. 1994.
\bibitem{derksen}
   H. Derksen,
   ``Error-Correcting Codes and $ B_h $-Sequences,''
   \emph{IEEE Trans. Inform. Theory}, vol. 50, no. 3, pp. 476--485, Mar. 2004.
\bibitem{fejes-toth}
   G. Fejes T\'{o}th,
   ``New Results in the Theory of Packing and Covering,''
   in P. M. Gruber and J. M. Wills (Eds.), \emph{Convexity and Applications}, Birkh\"{a}user, 1983, pp. 318--359.
\bibitem{golomb+welch}
   S. W. Golomb and L. R. Welch,
   ``Perfect Codes in the Lee Metric and the Packing of Polyominoes,''
   \emph{SIAM J. Appl. Math.}, vol. 18, no. 2, pp. 302--317, Mar. 1970.
\bibitem{groemer}
	 H. Groemer,
   ``Continuity Properties of Voronoi Domains,''
   \emph{Monatsh. Math.}, vol. 75, no. 5, 423--431, Oct. 1971.
\bibitem{gruber}
   P. M. Gruber,
   \emph{Convex and Discrete Geometry},
   Springer, 2007.
\bibitem{gruber+lekk}
   P. M. Gruber and C. G. Lekkerkerker,
   \emph{Geometry of Numbers},
   2nd ed., North--Holland, 1987.
\bibitem{sequences}
   H. Halberstam and K. F. Roth,
   \emph{Sequences},
   Springer-Verlag, 1983.
\bibitem{hoylman}
   D. J. Hoylman,
   ``The Densest Lattice Packing of Tetrahedra,''
   \emph{Bull. Amer. Math. Soc.}, vol. 76, no. 1, pp. 135--137, 1970.
\bibitem{jia}
   X.-D. Jia,
   ``On Finite Sidon Sequences,''
   \emph{J. Number Theory}, vol. 44, no. 1, pp. 84--92, May 1993.
\bibitem{kovacevic}
   M. Kova\v{c}evi\'c,
   ``Codes in $ A_n $ Lattices:\ Geometry of $ B_h $ Sets and Difference Sets,''
   preprint available at arXiv:1409.5276v4 [math.CO].
\bibitem{obryant}
   K. O'Bryant,
   ``A Complete Annotated Bibliography of Work Related to Sidon Sequences,''
   \emph{Electron. J. Combin.}, \#DS11, 39 pp. (electronic), 2004.
\bibitem{rogers}
   C. A. Rogers,
   \emph{Packing and Covering},
   Cambridge University Press, 1964.
\bibitem{roth}
   R. M. Roth,
   \emph{Introduction to Coding Theory},
   Cambridge University Press, 2006.
\bibitem{sidon}
   S. Sidon,
   ``Ein Satz \"{u}ber Trigonometrische Polynome und Seine Anwendung in der Theorie der Fourier-Reihen'' (in German),
   \emph{Math. Ann.}, vol. 106, no. 1, pp. 536--539, 1932.
\bibitem{singer}
   J. Singer,
   ``A Theorem in Finite Projective Geometry and Some Applications to Number Theory,''
   \emph{Trans. Amer. Math. Soc.}, vol. 43, pp. 377--385, 1938.
\bibitem{stein84}
   S. Stein,
   ``Packings of $ R^n $ by Certain Error Spheres,''
   \emph{IEEE Trans. Inform. Theory}, vol. 30, no. 2, pp. 356--363, Mar. 1984.
\bibitem{stein+szabo}
   S. Stein and S. Szab\'{o},
   \emph{Algebra and Tiling: Homomorphisms in the Service of Geometry},
   The Mathematical Association of America, 1994.
\bibitem{xue+zong}
   F. Xue and C. Zong,
   ``On Lattice Coverings by Simplices,''
   preprint available at arXiv:1502.04508 [math.MG].

\end{thebibliography}

\end{document}